\documentclass[12pt, a4paper]{article}

\usepackage{a4wide}
\usepackage[english]{babel}
\usepackage[utf8x]{inputenc}

\usepackage{amsmath,amsfonts,amssymb,amsthm} 
\usepackage{graphicx} 

\usepackage{euscript} 
\usepackage{amsfonts}
\usepackage{ mathrsfs }
\usepackage{tikz}
\usepackage{amscd}
\usepackage[alphabetic]{amsrefs} 
\usepackage[all]{xy}

\newtheorem{theorem}{Theorem}

\newtheorem{proposition}{Proposition}

\newtheorem{corollary}{Corollary}
\theoremstyle{definition}
\newtheorem{definition}{Definition}
\newtheorem{remark}{Remark}

\renewcommand{\leq}{\leqslant}

\DeclareMathOperator{\Aut}{Aut}
\DeclareMathOperator{\Stab}{Stab}

\DeclareMathOperator{\skd}{sk.dim}
\DeclareMathOperator{\Gal}{Gal}
\DeclareMathOperator{\Spec}{Spec}

\newenvironment{enumerate*}%
  {\begin{enumerate}%
    \setlength{\itemsep}{1pt}%
    \setlength{\parskip}{1pt}}%
  {\end{enumerate}}

\title{Geometry of symmetric spaces of type EIII}

\author{Viktor A. Petrov\footnote{St. Petersburg State University, St. Petersburg, Russia}\ \footnote{\texttt{victorapetrov@googlemail.com}}, Andrei V. Semenov\footnote{St. Petersburg State University, St. Petersburg, Russia}\ \footnote{\texttt{asemenov.spb.56@gmail.com}}}

\begin{document}

\maketitle

\begin{abstract}
    In this paper we generalize Atsuyama's result on the geometry of symmetric spaces of type EIII to the case of arbitrary fields of characteristic not 2 or 3. As an application we prove a variant of ``Chain lemma'' for microweight tori in groups of type $E_6$.
    
\end{abstract}

\section{Introduction}
Symmetric spaces are one of the objects of main interest in differential geometry. Boris Rosenfeld noticed that many symmetric spaces can be realized as elliptic planes over the tensor product of two composition algebras (see \cite{Ros}). Vinberg and Atsuyama counted the number of lines through two points in the general position independently \cites{A,V}. Moreover, Atsuyama identified the variety of the lines through two points in the special position for symmetric spaces of types EIII, EVI, and EVIII (see \cite{A}).

It is important that one can develop symmetric spaces not only over $\mathbb{R}$ or $\mathbb{C}$, but over any field of characteristic not two. In this case they are related to involutions on simple algebraic groups, see the papers by Helminck, Richardson, Springer (\cites{Hel, R, S} respectively) and Hutchens and Hunnel (\cites{H, Hu}).

Our main goal is to generalize Atsuyama's result to the case of arbitrary fields. Note that the partial case of the space EIII with a split group of type $E_6$ was considered by Veldkamp and Springer (see \cite{SV}) even before works of Atsuyama. To achieve the result they used direct calculations in the Albert algebra. We obtain a generalization of their results in the case of the space EIII for an arbitrary group of type $E_6$ with trivial Tits algebras and for an arbitrary field of characteristic not 2 and not 3. We use mostly geometric considerations (namely, the Chernousov--Merkurjev filtration on the product of two projective homogeneous varieties and geometry of subalgebras in non-associative algebras) and avoid lengthy calculations.

As an application we prove a variant of ``Chain lemma'' for microweight tori in a group of type $E_6$. We intend to use this lemma to constuct new cohomological invariants of Brown algebras in the same spirit as it is done in \cite{PR} for Albert algebras.



\section{Generalities}

\subsection{Albert algebras}
Let $F$ be a field of characteristic not 2 or 3.

\begin{definition}
An {\it Albert algebra} $A$ over a field $F$ is a central simple exceptional Jordan algebra, such that $\dim_F A = 27$.
\end{definition}

For every such $A$ one can find a cubic norm map $N: A \to F$ and a linear trace map $T: A \to F$. The trace map induces a non-degenerate symmetric bilinear form $T(x,y):= T(xy)$ on $A$.

For any cubic map $f: X \to Y$ one can write
$$
f (\sum_i t_i x_i) = \sum t_i^3 f(x_i) + \sum_{i \not= j} t_i^2t_j f(x_i,x_j) + \sum_{i < j< k} t_i t_j t_k f(x_i,x_j,x_k),
$$
so for any $x$ we can express $N(x+ty)$ in this way and obtain a linear map $f_x : A \to F$ by considering the coefficient of $t$ for a fixed $x \in A$.
\begin{definition}

For every $x\in A$ we define an element $x^{\#} \in A$ such that $T(x^{\#}, y) = f_x(y)$ for all $y \in A$. Note that such an element
exists and unique since $T$ is nondegenerate.
\end{definition}

It is easy to see that the map $\#: A \to A$ is quadratic, so we can define its linearization $\times : A\times A \to A$.
\begin{definition}
The {\it Freudenthal cross product} is a linearization of the map $\#$, given by the formula 
$$
x \times y = (x+y)^{\#} - x^{\#} - y^{\#}.
$$
\end{definition}




\subsection{Brown algebras}
Let $A$ be an Albert algebra over $F$. Consider a structurable algebra $B(A, F \times F)$ by setting $B$ to be the vector space $\begin{pmatrix} F, & A \\ A, & F \end{pmatrix}$ with multiplication map defined by
$$
\begin{pmatrix} \alpha_1, & j_1 \\ j_1', & \beta_1 \end{pmatrix} \cdot \begin{pmatrix} \alpha_2, & j_2 \\ j_2', & \beta_2 \end{pmatrix} = \begin{pmatrix} \alpha_1 \alpha_2 + T(j_1,j_2'), & \alpha_1j_2 +\beta j_1 + j_1' \times j_2' \\ 
\alpha_2 j_1' + \beta_1 j_2' + j_1 \times j_2, & \beta_1\beta_2 + T(j_2, j_1') \end{pmatrix}. 
$$
Also define the involution on $B$ by the formula $\overline{\begin{pmatrix} \alpha, & j_1 \\ j_1', & \beta \end{pmatrix}} = \begin{pmatrix} \beta, & j_1 \\ j_1', & \alpha \end{pmatrix}$. This leads us to a central simple structurable algebra over $F$ (see \cite{G}).


\begin{definition}
Let $F_s$ be the separable closure of $F$. A structurable algebra $(B,-)$ is said to be a {\it Brown algebra} if $(B,-) \otimes F_s \simeq B^d \otimes F_s$, where $B^d = B(A^d, F \times F)$ is the split Brown algebra. 
\end{definition}
The space of skew-symmetric elements is one-dimensional (we write $\skd B =1$) and is spanned by a single element $s_0 \in B$ such that $s_0^2 \in F^*$. We call $B$ a Brown algebra of type 1, if $s_0^2$ is a square in $F$, and we call $B$ a Brown algebra of type 2 otherwise.
\begin{proposition}[Theorem 2.9 in \cite{G}]
If $B$ is a Brown algebra of type $i\in \{1,2\}$ over $F$, then the connected component $\Aut^+(B)$ of the group of automorphisms is a simple simply connected group of type ${}^i E_6$ with trivial Tits algebras, and any such group can be obtained this way.
\end{proposition}

\subsection{Weil restriction}
Let $S_1, S_2$ be schemes and let $f: S_1 \to S_2$ be a morphism of schemes. 
For an $S_1$-scheme $X_1$ consider the functor 
$$R_{S_1/S_2} (X_1): (Sch/S_2)^{op} \longrightarrow Sets,$$
such that $R_{S_1/S_2} (X_1) (T) = X_1(T \times_{S_2} S_1)$.

\begin{definition}
If $R_{S_1/S_2} (X_1)$ is representable by an $S_2$-scheme $X_2$ we say that $X_2$ is a {\it Weil restriction} of $X_1$ along $f$.
\end{definition}

By definition the Weil restriction functor $R_{S_1/S_2}$ is adjoint to the fiber product functor $T\mapsto T\times_{S_2}S_1$.

In the case of $S_1 = \Spec(L)$ and $S_2 = \Spec(F)$ we simply write $R_{L/F}$ for a field $F$ and its finite extension $L$. If $X$ is an affine $L$-variety defined by
$$
X = \Spec \big(L[x_1, \dots, x_n]/ (f_1, \dots, f_m) \big),
$$
we can represent $R_{L/F}(X)$ by $\Spec \big( F[y_{i,j}]/(g_{k,l}) \big)$, where $x_i = \sum_j y_{i,j} e_j$ and $f_k = \sum_l g_{k,l} e_l$ for an $F$-basis $\{e_i\}$, $i=1,\ldots d$, of $L$.

\section{Description of EIII}
Let $F$ be a field of characteristic not $2$ or $3$ and let $B$ be a Brown algebra over $F$ with skew-hermitian elements spanned by $s_0$. Denote by $K$ the quadratic extension of $F$ generated by the square root of $s_0^2$ (as an element of $F$). As above, by
$G=\Aut(B)^+$ we denote the connected component of the automorphisms group of $B$, which is of type ${}^2E_6$. $B_K$ becomes of type $1$ over $K$ and so it becomes isomorphic to $B(A,K\times K)$ for some Albert algebra $A$ over $K$.

Recall that a semisimple group $G$ is called {\bf isotropic} if there exists $\mathbb{G}_m \leq G$ (or, which is the same, if there exists a proper parabolic subgroup $P \leq G$ defined over $F$) and {\bf anisotropic} otherwise.

Let $\mathcal{Q}$ be the set of all quaternion subalgebras (in the sense of algebras with involution) in $B$.

\begin{proposition}\label{prop:Points}
The set $\mathcal{Q}$ can be viewed as the set of the $F$-rational points of an affine open subvariety
of the Weil restriction $R_{K/F}(G_K/P_1)$, where $P_1$ is a maximal parabolic subgroup of type $1$. This subvariety is a twisted form of {\mbox{$E_6/D_5\cdot{\mathbb G}_m$}}. Moreover, if $G$ is anisotropic, then the complement to this open subvariety has no $F$-rational points.
\end{proposition}
\begin{proof}
Since $\skd B = 1$, the involution on a quaternion subalgebra $H \in \mathcal{Q}$ must be orthogonal, so $H$ must contain $K$. Let us present $H=K\oplus Kx$, where $x$ is orthogonal to $1$ with respect to the trace form, and $x^2\in F$. Setting over $K$
$$
x=\begin{pmatrix}0&e\\e'&0\end{pmatrix}
$$
we see that $e\times e=0$, $e'\times e'=0$ and $T(e,e')\ne 0$. Since $Kx$ is stable under the action of the Galois group of $K$ over $F$, we see that $Fe'=F\sigma(e)$, where $\sigma$ is the (unique) nontrivial element of $\Gal(K/F) = \mathbb{Z}/2\mathbb{Z}$. Now $e\times e=0$ means that $\Stab(Ke)$ is a parabolic subgroup $P$ of type $P_1$ in $G_K$ (see, for example, \cite[Theorem~7.2]{G}), and the condition $T(e,\sigma(e))\ne 0$ is open and means that the image $\sigma(P)$ of $P$ is opposite to $P$. Now $G(K)$ acts transitively on the pairs of opposite parabolic subgroups. A stabilizer of this action is equal to the Levi subgroup of $P_1$, which is of type $D_5$.

If the complement has a rational point, then over $K$ there exists a parabolic subgroup $P$ of type $P_1$ such that $\sigma(P)$ is not opposite to $P$. Then $P\cap\sigma(P)$ is defined over the base field and by \cite[Expos\'e XXVI, Th\'eor\`em 4.3.2 (iv)]{DG} contains unipotent elements. Since the characteristic of the base field is neither $2$ nor $3$, the main result of \cite{T} tells us that $G$ is isotropic.
\end{proof}

\begin{remark}\label{rem:tori}
It follows from the proof that $\mathcal{Q}$ has another description as the variety of tori of the form $R_{K/F}^{(1)}(\mathbb{G}_m)$ inside $\Aut(B)^+$, such that their centralizers become Levi subgroups of parabolic subgroups of type $P_1$ over $K$ (following \cite{VN} we call such tori {\bf microweight}). Note that every such torus contains a unique nontrivial involution and can be reconstructed by this involution as the center of its fixed point subgroup. This relates our work with \cite{H}, cf., for example, \cite[Theorem~6.11(1)]{H}.
\end{remark}
Now consider the map $\pi$ that sends a quaternion subalgebra $H \in \mathcal{Q}$ to $K^\perp_H\cdot B$, where $K^\perp_H$ stands for the orthogonal complement to $K$ in $H$, and the product is point-wise.
\begin{proposition}\label{prop:Dual}
$\pi(H)$ is a $22$-dimensional subalgebra of $B$ for any $H \in \mathcal{Q}$. Moreover, $\pi$ defines a bijection between the open subvarieties of $R_{K/F}(G_K/P_1)$ and $R_{K/F}(G_K/P_6)$.
\end{proposition}
\begin{proof}
Using the notation from the proof of Proposition~\ref{prop:Points}, for any $H \in \mathcal{Q}$ we can associate a pair $(e,e') \in A\times A$. Now we see that over $K$
$$
\pi(H)_K=\begin{pmatrix}K& e'\times A\\e\times A&K\end{pmatrix}.
$$
The vector space $e\times A$ is known to be $10$-dimensional over $K$ (see, for example, \cite[Lemma~6.6]{G}). Its stabilizer under the action of $G$ is a parabolic subgroup of type $P_6$ (\cite[Theorem~7.2]{G}) and $Ke$ can be reconstructed from $e\times A$ (\cite[Lemma~6.7]{G}).
\end{proof}

Let us call quaternion subalgebras {\bf points} and $22$-dimensional subalgebras as in Proposition~\ref{prop:Dual}  {\bf lines}. We say that a point $H$ is {\bf incident} to a line $L$ if $H\le L$. In terms of the tori description this means that the corresponding tori $R_{K/F}(\mathbb{G}_m)$ (and the corresponding involutions) commute but do not coincide.

\section{Chernousov--Merkurjev filtration for $E_6$}
Let $G$ be a group of inner type $E_6$ over $K$, possessing parabolic subgroups of types $P_1$ and $P_6$. Let $P_{i,j}$ be the submaximal parabolic subgroup of $G$ (meaning that any proper overgroup of $P_{i,j}$ must coincide with $P_i$ or $P_j$) for $1 \leq i \not= j \leq 6$.
\begin{proposition}\label{prop:CM}
There are filtrations whose consecutive complements are affine bundles over projective homogeneous varieties, as shown on the picture:
$$
\xymatrix{
{G/P_6\times G/P_6}&{\mbox{\Large$\supset$}}\ar[d]^-{\mathbb{A}^8}&X&{\mbox{\Large$\supset$}}\ar[d]^-{\mathbb{A}^1}&{G/P_6}\\
&{G/P_{1,6}}&&{G/P_{5,6};}
}
$$
$$
\xymatrix{
{G/P_6\times G/P_1}&{\mbox{\Large$\supset$}}\ar[d]^-{\mathbb{A}^{16}}&Y&{\mbox{\Large$\supset$}}\ar[d]^-{\mathbb{A}^5}&{G/P_{1,6}}\\
&{G/P_6}&&{G/P_{5,6}.}
}
$$
The affine bundle maps are given by the rule $(P, Q) \mapsto (P\cap Q)\cdot R_u(P)$ (where $P$ is of type $P_6$ and $Q$ is of type $P_6$ or $P_1$ respectively, and $R_u$ stands for the unipotent radical).
\end{proposition}
\begin{proof}
This is a particular case of \cite[\S~4 and \S~5]{CM} that asserts that the $G$-orbits on $G/P\times G/Q$ are in one-to-one correspondence with the double cosets $W_P{\backslash}W/W_Q$ and provides the affine bundle stratification as above. To make computations of the double cosets explicit one can use pictures from \cite{PSV}. Namely, one can take the Hasse diagram of $E_6/P_6$ (which is the same as the Hasse diagram of the weights of the representation $V(\varpi_6)$ since the latter is minuscule) and cut off the edges labelled by $6$ (respectively $1$).
\end{proof}

The Proposition implies that there are three possible mutual positions of lines $L_1, L_2$ (or a points respectively): they can coincide (which means $L_1=L_2 \in R_{K/F}(G_K/P_6)$), be in the general position (which means $(L_1,L_2) \in R_{K/F}((G_K/P_6 \times G_K/P_6) \setminus X)$) or in what we call the special position (which means $(L_1,L_2) \in R_{K/F}(X)$). Similarly, there are three possible mutual positions of a point and a line: they can be incident, in the general position or in what we call the special position. This agrees with ``Hjemslev--Moufang'' description as in \cite{SV} in the split case. Namely, the lines (or points) in the special position are called there ``connected'', and the corresponding condition on lines $K e_1$ and $K e_2$ reads as follows: $e_1\times e_2=0$. Further, a point and a line in the special position are also called ``connected'', and the condition is $T(e_1,e_2)=0$. Finally, a point and a line are incident if and only if $\langle e_1,e_2\rangle=0$, where $\langle \cdot,\cdot\rangle$ is a transformation defined by Freudenthal.

\section{Main theorem}

\begin{theorem}\label{thm:main}
\begin{enumerate}
\item If $L_1$ and $L_2$ are two lines in the general position, then they meet at at most one point. The condition that they meet at exactly one point is open, and if $G$ is anisotropic, the complement to this open subvariety has no $F$-rational points.

\item If $L_1$ and $L_2$ are lines in the special position, then the set $\{H \in \mathcal{Q} \mid H \le L_1 \cap L_2\}$ can be viewed as the set of the $F$-rational points of an affine open subvariety of $R_{K/F}({\mathbb P}^4)$, which is a twisted form of $A_4/A_3\cdot{\mathbb G}_m$.

\item For any line $L$ the set of all points on $L$ can be viewed as
the set of the $F$-rational points of an affine open subvariety of the Weil restriction from $K$ to $F$ of an $8$-dimensional isotropic smooth quadric, which is a twisted form of $D_5/D_4\cdot{\mathbb G}_m$.
\end{enumerate}
\end{theorem}
\begin{proof}
Let $L_1, L_2$ be two lines. According to the proof of Proposition~\ref{prop:Dual} there exists $(e^{(i)}_1, e_2^{(i)}) \in A\times A$, such that 
\[
(L_i)_K=\begin{pmatrix}K & e^{(i)}_1 \times A \\ 
e_2^{(i)} \times A& K\end{pmatrix},
\]
so each $L_i$ corresponds to the uniquely defined $10$-dimensional subspace $V_i$ over $K$. It is known (see \cite{G}) that their stabilizers $Q_i=\Stab(V_i)$ are parabolic subgroups of type $P_6$ in $G(K)$.
Formulae from \cite[\S~1, $5^\circ$]{PSV} imply that the unipotent radicals $R_u(Q_i)$ act on $V_i$ trivially.

Note that $V_1\cap V_2\ne 0$. In order to prove the claim one may check it over $\overline{F}$, since being equal to zero is a condition that preserves under the Galois descent. So we may assume that $G$ is split. In this case the assertion follows from \cite[Proposition~3.3 and Proposition~3.5]{SV}.

Now define $Q=(Q_1\cap Q_2)R_u(Q_1)$. According to Proposition~\ref{prop:CM}, if $L_1$ and $L_2$ are in the general position then $Q$ is a parabolic subgroup of type $P_{1,6}$, and if $L_1$ and $L_2$ are in the special position, then $Q$ is of type $P_{5,6}$. So $Q$ is submaximal in the sense that any proper overgroup of $Q$ must coincide with $Q_1$ (of type $P_6$) or be a parabolic subgroup of type $P_1$ (in the case of general position) or $P_5$ (in the case of special position).

On the other hand, $Q$ stabilizes the flag $V_1\cap V_2\le V_1$. The description of parabolic subgroups in \cite[Theorem~7.2]{G} implies that $V_1\cap V_2$ is $1$-dimensional (resp. $5$-dimensional) and totally isotropic (that is the restriction of $\times$ on $V_1\cap V_2$ is trivial). \par

Note that $V_1\cap V_2$ is $\sigma$-stable, since $L_1$ and $L_2$ are both defined over the base field $F$. So there exists an $F$-space $V$ such that $V_1\cap V_2=V_K$. Now observe that $K\oplus V$ is a subalgebra of $B_F$, since 
$$
(K\times K)\oplus V_i = \begin{pmatrix}K & e^{(i)}_1 \times A \\ 
e_2^{(i)} \times A& K\end{pmatrix}
$$
is a subalgebra of $B_K$ for $i=1,2$. The set $\{H \in \mathcal{Q} \mid H \in L_1 \cap L_2\}$ is the set of the $F$-rational points variety of quaternion subalgebras of this structurable algebra (of dimension $4$ in the case of general position, of dimension $12$ in the case of special position, and of dimension $22$ when $L_1=L_2$).

In the case of dimension $4$ the subalgebra over $K$ looks like
$$
\begin{pmatrix}
K&Ke\\
K\sigma(e)&K
\end{pmatrix},
$$
and the same argument as in the proof of Proposition~\ref{prop:Points} shows that the condition that this algebra is quaternionic is open, and the complement to this open subvariety has no $F$-rational points if $G$ is anisotropic.

In the case of dimension $12$ the automorphisms group of the algebra is of type $A_4$, since the cross product is trivial and henceforth this automorphisms group consists only of authomorphisms of a $5$-dimensional vector space.

Finally, in the case of dimension $22$ the group $Q$ is equal to $Q_1$, so it has Levi subgroup of type $D_5$, which is isotropic (over $K$). Repeating the proof of Proposition~\ref{prop:Points} with appropriate changes we get the claim.
\end{proof}

\begin{corollary}[Chain Lemma]
Assume that $F$ is infinite. Let $T$ and $T'$ be two microweight tori in $G$ (see Remark~\ref{rem:tori}). Then there exists a sequence of microweight tori $T_i$'s such that $T_0=T$, $T_n=T'$, and $T_i$ commutes with $T_{i+1}$ for any $i=0,\ldots,n-1$ (actually one can take $n=4$).
\end{corollary}
\begin{proof}
Let $H$ and $H'$ be the points corresponding to $T$ and $T'$. Note that an isotropic quadric contains an affine space as an open subvariety. It follows by the dual (in the sense of Proposition~\ref{prop:Dual}) version of Theorem~\ref{thm:main} (3) that there exists a parametrization of (some of the) lines passing through $H$ and lines passing through $H'$ by $F$-rational points of $R_{K/F}(\mathbb{A}^8)\simeq{\mathbb A}^{16}$. Now the condition of being in the general position is open; moreover, for two lines in the general position the condition that they meet at exactly one point is also open by Theorem~\ref{thm:main} (1). Since $F$ is infinite, there exist two lines $L$ and $L'$ passing through $H$ and $H'$ respectively such that they meet at some point $H''$. Now set $T_1$, $T_2$ and $T_3$ to be the tori corresponding to $L$, $H''$ and $L'$ respectively.
\end{proof}

This Corollary leads to an interesting geometry of microweight tori in the spirit of the paper of Vavilov and Nesterov \cite{VN}.

\subsection*{Acknowledgements}

We are grateful to Nikolai Vavilov and Anastasia Stavrova for the attention to our work.

\subsection*{Funding}
Theorem~1 was obtained under the support by Russian Science Foundation grant 20-41-04401. The second author was supported by Young Russian Mathematics award and by ``Native Towns'', a social investment program of PJSC ``Gazprom Neft''. Also the second author is supported in part by The Euler International Mathematical Institute, grant number 075-15-2019-16-20.

\newpage
\thispagestyle{empty}
Viktor Alexandrovich Petrov: St. Petersburg State University, 199178 Saint Petersburg, Russia, Line 14th (Vasilyevsky Island), 29

\smallskip

\texttt{victorapetrov@googlemail.com}

\medskip

Andrei Vyacheslavovich Semenov: St. Petersburg State University, 199178 Saint Petersburg, Russia, Line 14th (Vasilyevsky Island), 29

\smallskip

\texttt{asemenov.spb.56@gmail.com}

\end{document}